\documentclass{amsart}

\usepackage{amsmath,amssymb,amsfonts,enumerate,amsthm,graphicx,color}
\usepackage{tkz-graph}
\usetikzlibrary{fadings}
\usetikzlibrary{patterns}
\usetikzlibrary{backgrounds}
\def\opn#1#2{\def#1{\operatorname{#2}}} 
\opn\chara{char} \opn\length{\ell} \opn\pd{pd} \opn\rk{rk}
\opn\projdim{proj\,dim} \opn\injdim{inj\,dim} \opn\rank{rank}
\opn\depth{depth} \opn\grade{grade} \opn\height{height}
\opn\embdim{emb\,dim} \opn\codim{codim}

\opn\Tr{Tr} \opn\bigrank{big\,rank}
\opn\superheight{superheight}\opn\lcm{lcm}
\opn\trdeg{tr\,deg}
\opn\reg{reg} \opn\lreg{lreg} \opn\ini{in} \opn\lpd{lpd} \opn\HS{HS}
\opn\size{size} \opn\sdepth{sdepth} \opn\Ass{Ass}\opn\LT{LT}
\opn\depth{depth} \opn\k{\bf{k}}
\newtheorem{Theorem}{Theorem}[section]
\newtheorem{Lemma}[Theorem]{Lemma}
\newtheorem{Corollary}[Theorem]{Corollary}
\newtheorem{Proposition}[Theorem]{Proposition}
\newtheorem{Remark}[Theorem]{Remark}

\newtheorem{Example}[Theorem]{Example}

\newtheorem{Definition}[Theorem]{Definition}

\let\epsilon\varepsilon
\let\kappa=\varkappa
\def\qed{\ifhmode\textqed\fi
      \ifmmode\ifinner\quad\qedsymbol\else\dispqed\fi\fi}
\def\textqed{\unskip\nobreak\penalty50
       \hskip2em\hbox{}\nobreak\hfil\qedsymbol
       \parfillskip=0pt \finalhyphendemerits=0}
\def\dispqed{\rlap{\qquad\qedsymbol}}

\opn\dis{dis}
\def\pnt{{\raise0.5mm\hbox{\large\bf.}}}

\opn\Lex{Lex}

\begin{document}

\title{On  $m$-closed graphs}


\author {  Leila Sharifan, Masoumeh Javanbakht   }

\address{ Department of Mathematics, Hakim Sabzevari University, Sabzevar, Iran, P.O. Box 397}
\email {leila-sharifan@aut.ac.ir, \ masumehjavanbakht@gmail.com}


\begin{abstract}
A graph is closed when its vertices have a labeling by $[n]$ such
that the binomial edge ideal $J_G$ has a quadratic Gr\"{o}bner basis
with respect to the lexicographic order induced by $x_1
> \ldots
> x_n > y_1> \ldots > y_n$. In this paper, we generalize this notion
and study the so called $m-$closed graphs. We find equivalent
condition to $3-$closed property of an arbitrary  tree $T$. Using
it, we classify a class of $3-$closed trees. The primary
decomposition of this class of graphs is also studied.
\end{abstract}

\maketitle
 \section{Introduction and Preliminaries}
 Suppose $G$ is a simple graph on the vertex set $[n]$ and
$R = \k[x_1,\ldots, x_n, y_1,\ldots, y_n]$ is the polynomial ring
over the field $\k$. The {\em binomial edge ideal} of $G$ is the
ideal
\[
J_G=(f_{ij}:\ \{i,j\}\in E(G)\ {\rm and}\ i<j)\subset R,
\]
where $f_{ij} = x_iy_j -x_jy_i$. This notion was first introduced in
\cite{HHKR} and independently in \cite{O}. \vspace{2mm}

Note that any ideal generated by a set of $2$-minors of a $2 \times
n$-matrix $X$ of indeterminates may be viewed as the binomial edge
ideal of a graph. In \cite{HHKR}, the authors compute the reduced
Gr\"{o}bner basis of the binomial edge ideal with respect to the
lexicographic order induced by $x_1>\cdots>x_n> y_1>\cdots >y_n$ (we
show this order by $\prec$). In particular, they find the necessary
and sufficient conditions in which $J_G$ has a quadratic Gr\"{o}bner
basis. Graphs whose binomial edge ideal has a quadratic Gr\"{o}bner
basis are called closed graphs and the Cohen-Macaulay property of
these graphs is studied in \cite{EHH}.  Recently, many authors
studied the algebraic properties of some classes of binomial edge
ideals. In particular the regularity and the depth  are studied in
\cite{ChDI,EHH,EZ,SK2,MM,SK}. But the reduced Gr\"{o}bner basis
obtained in \cite{HHKR} has not been studied in more details.

In this paper, we study the  Gr\"{o}bner basis of $J_G$ where $G$ is
a simple graph. We call $G$ an $m-$closed graph when its vertices
can be labeled by $[n]$ such that the elements of the reduced
Gr\"{o}bner basis of $J_G$ have degree at most $m$, and $m$ is the
least integer with this property for $G$. Note that by this
definition, a closed graph is a $2-$closed graph.

In Section 2 we study some basic properties of $m-$closed graphs. In
particular, we show that a cycle $C_n$  ($n>3$) is $m-$ closed where
$m=\left\{
\begin{array}{ll}
    \frac{n}{2}+1 & \hbox{$n$ is even;} \\
    \frac{n+1}{2}+1 & \hbox{$n$ is odd.} \\
\end{array}%
\right.$ (see Theorem \ref{cycle}).
 Using it we conclude that in each
$m-$closed graph, any cycle with at least $2m-1$ vertices has a
chord.

The notion of weakly closed graphs has been introduced in \cite{M1}
as a generalization of closed graphs. The final result of section 2
shows that each weakly closed graph is $m-$closed for some $m\leq 4$
 (see Theorem \ref{weakly closed}).

In Section 3 we study $3-$closed property of trees and we show that
a tree $T$ with $n$ vertices is $3-$closed if and only if  it is not
a path and there exists a labeling of its vertices such that
$d(i,i+1)\leq 2$ for each $i<n$ (see Theorem \ref{3-closed trees}).
The class of $3-$closed trees and the number of elements of the
reduced  Gr\"{o}bner basis of $J_T$ for a $3-$closed labeling is
also studied by means of the   bipartite graph  $G^*$ attached to a
simple graph $G$ corresponding to the generators of $J_G$ (see
Definition \ref{bipartite}, Theorem \ref{3-closed class} and
Corollary \ref{number of reduced}).

In Section 4, we study a class of trees constructed from caterpillar
trees. We characterize the minimal primary decomposition of this
class of trees (see Theorem \ref{primary of caterpillar}). Also, we
show that they are $3-$closed. For some other trees constructed by
caterpillar trees we show $3-$closed property (see Theorem
\ref{caterpilar trees}). To prove Theorem \ref{caterpilar trees}, we
need an algorithm to give a $3-$closed labeling to the vertices of a
caterpillar tree such that $1$ is assigned to an arbitrary vertex.
This is provided in Algorithm 1 presented in the Appendix section.

\vspace{2mm}

 In the following, we review some definitions and results from
\cite{HHKR} which we need in the next sections.

\begin{Definition}
 Let $G$ be a simple graph on $[n]$, and let $i$ and $j$ be two
vertices of G with $i < j$. A path $i = i_0, i_1,\ldots,
i_r = j$ from $i$ to $j$ is called {\it admissible}, if\\
(i) $i_k\neq i_\ell$ for $k\neq \ell$;\\
 (ii) for each $k = 1,\ldots
, r-1$ one has either $i_k < i $ or $i_k
> j$;\\
 (iii) for any proper subset $\{ j_1,\ldots , j_s\}$ of $\{i_1,\ldots,
i_{r-1}\}$, the sequence $\{i, j_1,\ldots, j_s, j\}$ is not a path.
\end{Definition}

 Given an admissible path $\pi: i = i_0, i_1, \ldots , i_r = j$ from $i$ to $j$,
  where $i < j$, we associate the monomial

  $$u_\pi=(\prod_{i_k>j}x_{i_k})(\prod_{i_\ell<i}y_{i_\ell}).$$
  By \cite[Chapter 2, Proposition 6]{CLO}, the reduced Gr\"{o}bner basis of
  $J_G$ with respect to $\prec$ is unique. We have:

\begin{Theorem}\label{reduced grobner basis}\cite[Theorem 2.1]{HHKR}\\
Let $G$ be a simple graph on $[n]$. Then the set of binomials
$$\mathcal{G} =\bigcup_{i<j}\{u_\pi f_{i j} : \ \pi \ {\text {is an
admissible path from}}\  i \ {\text {to}} \ j \}$$ is the reduced
Gr\"{o}bner basis of $J_G$.
\end{Theorem}

\vspace{2mm}

By \cite[Theorem 3.2]{HHKR}, we can  write $J_G$ as a finite
intersection of prime ideals.  In fact, corresponding to each subset
$S\subset[n]$ we have the prime ideal
\[
P_S(G)=(\bigcup_{i\in
S}\{x_i,y_i\})+J_{\tilde{G}_1}+\cdots+J_{\tilde{G}_{c(S)}},
\]
where $G_1,\ldots,G_{c(S)}$ are the connected components of the induced subgraph on the vertices $[n]\setminus S$,
and $\tilde{G}_\ell$ is the complete graph on the vertices of $G_\ell$ for all $\ell$.
Then
\begin{equation}\label{PD}
J_G= \bigcap_{S\subset[n]}P_S(G).
\end{equation}
Moreover, $\dim R/J_G = \max\{(n - |S|) + c(S) : S \subset [n]\}$
and hence  $\dim R/J_G\geq  n + c(G)$, where $c(G)$ is the number of
the connected components of $G$.  Equation (\ref{PD}) also shows
that $J_G$ is a radical ideal. If $G$ is a connected graph then
$P_\emptyset (G)=J_{K_n}$ is a minimal prime ideal of $J_G$. Note
that if $S$ is an arbitrary subset of $[n]$ the prime ideal $P_S(G)$
is not necessary a minimal prime ideal of $J_G$. The next lemma
detects the minimal prime ideals of $J_G$ when $G$ is a connected
graph. Note that for $S\subset [n]$, by $c(S)$ we mean
$c(G_{[n]\setminus S})$.

\begin{Lemma}\label{minimal prime}\cite[Corollary 3.9]{HHKR}

Let $G$ be a connected graph on the vertex set $[n]$ and $S\subset
[n]$. Then $P_S(G)$ is a minimal prime ideal of $J_G$ if and only if
$S = \emptyset$, or $S\neq \emptyset$ and for each $i\in S$ one has
$c(S \backslash \{i\}) < c(S)$.
\end{Lemma}
\section{$m-$closed graphs}

In this section we study the reduced  Gr\"{o}bner basis of $J_G$. As
Theorem \ref{reduced grobner basis} shows the reduced Gr\"{o}bner
basis depends on the labeling of the vertices of $G$. We recall that
a labeling of $G$ is a bijection $V(G)\simeq [n]=\{1,\ldots ,n\}$,
and given a labeling, we typically assume $V(G)=[n]$.

The graph $G$ is called closed with respect to the given labeling if
$J_G$ has a quadratic Gr\"{o}bner basis with respect to $\prec$. By
\cite[Theorem 1.1]{HHKR} we have:

\begin{Theorem}\label{Closed condition}

Let $G$ be a simple graph on the vertex set $[n]$.  $G$ is closed if
and only if  the following condition is satisfied:

\vspace{2mm}

For every two edges $\{i, j\}$ and $\{k, \ell\}$ in $E(G)$ with $i <
j$ and $k < \ell$, one has $\{ j, \ell\}\in E(G)$ if $i = k$, and
$\{i,k\}\in E(G)$ if $j = \ell$.

\end{Theorem}

Let $G$ be a graph, we recall that the {\it clique complex} of $G$,
denoted  $\Delta(G)$, is the simplicial complex  on $[n]$ whose
faces are the cliques of $G$.
 The graph $G$ is closed if and only if
there exists a labeling of $G$ such that all facets of $\Delta(G)$
are intervals $[a,b]\subset[n]$ (see \cite[Theorem 2.2]{EHH}).
Closed graphs are studied in more details in \cite{CE,CR}.


Following the definition of closed graph we introduced $m-$closed
graphs.

\begin{Definition}
 Let $m$ be a positive integer. We say
that a graph $G$ with vertex set $V (G) = \{v_1,\ldots, v_n\}$ is
$m-$closed, if its vertices can be labeled by $[n]$ such that for
this labeling  all the elements of $\mathcal{G}$ are of degree $\leq
m$, and $m$ is the least integer with this property for $G$.

Moreover, a labeling of the vertices of $G$ is called an $m-$closed
labeling if the reduced  Gr\"{o}bner basis of $J_G$ is in degree
$m$ and less than $m$ with respect to this labeling.
\end{Definition}

\vspace{2mm}

 By the above  definition a closed graph is a $2-$closed
graph. the cycle $C_4$ (cycle with $4$ vertices) is $3-$closed and
$C_5$ is $4-$closed.

By Theorem \ref{reduced grobner basis}, a graph $G$ is $m-$closed if
and only if, there exists a labeling for its vertices such that each
admissible path in $G$ has at most $m$ vertices  and
 in each labeling of the vertices, there exists an admissible path
of length $\ell$ where $\ell\geq m-1$.

\vspace{2mm}


We recall that a bridge is an edge  whose removal from a graph
increases the number of components. If $e$ is a bridge of a
connected graph $G$, and $H_1$ and $H_2$ are the connected
components of $G\setminus e$, we write $G\setminus e=H_1\sqcup H_2$.

In the following we  find some information about $m-$closed graphs.

\begin{Proposition}\label{bridge}
(i) Let $G$ be a graph and $\ell$ be the length of the longest
induced path of $G$. Then $G$ is $m-$closed for some $m\leq \ell+1$.

(ii) Let $G$ be a graph and $H$ be an $\ell-$closed induced subgraph
of $G$. Then $G$ is $m-$closed for some $m\geq \ell$.

(iii) Let $e$ be a bridge of a connected graph $G$ and $G\setminus
e=H_1\sqcup H_2$. If $H_1$ is $m-$closed and $H_2$ is $\ell-$closed
($\ell\geq m$), then $G$ is $\ell-$closed provided that there exists
an $m-$closed labeling  of $H_1$ in which $1$ is the label of the
end point of $e$ in $H_1$ and   there exists an $\ell-$closed
labeling of $H_2$ in which $1$ is the label of  the endpoint of $e$
in $H_2$.
\end{Proposition}
\begin{proof}
Part (i) and (ii) are followed from the definition of an admissible
path and $m-$closed property.

For part (iii), assume that $H_1$ is an $m-$closed graph on $[n_1]$,
$H_2$ is an $\ell-$closed  graph on $[n_2]$ and $1$ is the label of
the end points of $e$ in each $H_i$ ($i=1,2$). We give a labeling to
$G$ by  assigning  to each vertex $i$ of $H_1$ the new label
$n_1-i+1$ and to each vertex $i$ of $H_2$ the new label $n_1+i$. So,
by this labeling $e=\{n_1,n_1+1\}$. It is easy to see that the graph
$G=H_1\cup \{n_1,n_1+1\}\cup H_2$ is an $\ell-$closed graph on
$[n_1+n_2]$.
\end{proof}

 A natural question to ask is that if the reduced  Gr\"{o}bner basis of
$J_G$ has an element of degree $m$, can we conclude that it also has
an element of degree $\ell$ for each $1<\ell<m$. This is not true
in general, as the following example shows:

\begin{Example}
 Let $G$ be the path on $[5]$ with
$E(G)=\{\{1,4\},\{3,4\},\{3,5\},\{2,5\}\}$. Then $\mathcal{G}$ has
an element of degree $5$ while it doesn't have any element of degree
$4$.
\end{Example}

 For a simple graph $G$
on $[n]$, and $m\geq 3$, if the reduced
 Gr\"{o}bner basis of $J_G$ has an element of degree $m$, then it has an
element of degree $3$. In fact, $G$ is not closed and  by
 Theorem \ref{Closed condition}, there exist two edges $\{i,j\}$ and $\{i,\ell\}$ in $E(G)$  with $i<j$, $i<\ell$ and $\{j,\ell\}\notin E(G)$,
 or  there exist two edges $\{i,j\}$ and $\{k,j\}$ in $E(G)$  with $i<j$, $k<j$ and $\{i,k\}\notin
 E(G)$. So $j,i,\ell$ or $i,j,k$  is an admissible path of length
$2$. So, $\mathcal{G}$ has an element of degree $3$.

  Therefore,
if $G$ is an $m-$closed graph, in each labeling of its vertices,
there exists an admissible path of length $2$. But as we have seen
in the above example, we can not extend Theorem \ref{Closed
condition} to check if a labeling is a $3-$closed labeling or not.

%
%
We recall that if $I$ is an ideal of $R$, the leading term ideal of
$I$ with respect to $\prec$ is the monomial ideal of $R$ which is
generated by $(LT_{\prec}(f)\ | \ 0\neq f\in I)$ where
$LT_{\prec}(f)$ is the leading term of $f$ with respect to $\prec$ .
We write $LT_{\prec}(I)$ for the leading term ideal of $I$.

If $G$ is a graph, it is clear that for any arbitrary labeling of
the vertices of $G$, $|\mathcal{G}|=\mu(LT_\prec(J_G))\geq \mu(J_G)$
($\mu(I)$ is the minimal number of homogeneous generators of I).
Moreover, $G$ is a closed graph if and only if there exists a
labeling in which $\mu(LT(J_G))= \mu(J_G)$. So If $G$ is a
non-closed graph on $[n]$ and $\mu(LT_\prec(J_G))= \mu(J_G)+1,$ then
$G$ is $3$-closed.

\vspace{2mm}

It is well known by \cite[Proposition 1.2]{HHKR}, that a closed
graph is chordal. In the following we are going to find a
generalization of this necessary condition for $m-$closed property.
For this we need the following theorem about cycles:

\begin{Theorem}\label{cycle}
Let $C_n$ be the cycle on $n\geq 4$ vertices. Then $C_n$ is
$m-$closed where \\
$m=\left\{
\begin{array}{ll}
    \frac{n}{2}+1 & \hbox{$n$ is even;} \\
    \frac{n+1}{2}+1 & \hbox{$n$ is odd.} \\
\end{array}%
\right.$
\end{Theorem}
\begin{proof}
Let $C_n$ be the cycle on $n$ vertices and $m$ be as defined in the
theorem. To show the result, we first prove that in any labeling of
the vertices of $C_n$, one can find an admissible path with at least
$m$ vertices.

In an arbitrary labeling of the vertices of $C_n$, one of the
following situation happens:

{\bf case 1}: For all $i \in \{1,\ldots ,n-1\},\  d(i,i+1)=1$. This
case happens if and only if we give successive integers to the
vertices. i. e., ($E(C_n)=\{\{1,2\},\{2,3\},\ldots,
\{n-1,n\},\{n,1\}\}$). So $P: 1, n , n-1, \ldots , 3$ is an
admissible path with $n-1$ vertices and $n-1\geq m$.

{\bf case 2}: There exists $i \in \{1,\ldots ,n-1\},\
d(i,i+1)=\ell\geq 2$. So we have two admissible paths $$P_1: i,
j_1,\ldots , j_{\ell-1}, i+1\ {\text {and}} \ P_2: i, j'_1,\ldots ,
j'_{n-\ell-1}, i+1$$ where $\{i,i+1\}\sqcup \{j_1,\ldots ,j_{\ell
-1}\}\sqcup \{j'_1,\ldots , j'_{n-\ell-1}\}=[n]$, $P_1$ has $\ell+1$
vertices and $P_2$ has $n-\ell+1$ vertices.

In the case that $n$ is even, if $\ell+1<\frac{n}{2}+1$ and
$n-\ell+1<\frac{n}{2}+1$, then $n+2<n+2$ which is a contradiction.
So, one of the paths $P_1$ and $P_2$ has at least $m$ vertices.

Now assume that $n$ is odd. Since $d(i,i+1)=\ell$, we have $\ell\leq
n-\ell$. Moreover, $\ell=n-\ell$ if and only if $n=2\ell$ which is a
contradiction. So, $\ell<n-\ell$.

If $n-\ell+1<\frac{n+1}{2}+1$, then by $1+\ell<n-\ell+1<
\frac{n+1}{2}+1$ we have $n+2<n+2$ which is a contradiction. So
$P_2$ has at least $m$ vertices.

So in each labeling of the vertices of $C_n$, we have  an admissible
path with at least $m$ vertices.

Now, if we find a labeling of the vertices of $C_n$ such that each
admissible path has at most $m$ vertices, the conclusion follows.

Suppose that:

$$V(C_n)=\{v_1,v_2,\ldots,v_n\}, \ \
E(C_n)=\{\{v_1,v_2\},\{v_2,v_3\},\ldots
,\{v_{n-1},v_n\},\{v_n,v_1\}\}$$

If $n$ is even, we do as follows:

\begin{enumerate}
    \item $S=\{v_1,v_2,\ldots ,v_n\}$,
    \item label $v_1$ as $1$,
    \item $i=1$,
    \item \textbf{While} $i<n$ \textbf{do}
    \begin{enumerate}
        \item Pick $v_j\in S$ such that $d(i,v_j)=m-1$ and label
        $v_j$ as $i+1$,
        \item If $i+2<n$, label $v_{j+1}$ as $i+2$,
        \item $i:=i+2$.
    \end{enumerate}
\end{enumerate}
By this labeling of the vertices, for each $i$, $d(i,i+1)=m-1$ if
$i$ is odd and $d(i,i+1)=1$ if $i$ is even. So we have some
admissible path with $m$ vertices.

\vspace{2mm}

If $n$ is odd, we do as follows:

For each $1\leq i< m$, label $v_i$ as $2i-1$ and for each $m\leq
i\leq n$, label $v_i$ as $2(i-m+1)$. \\ By this labeling, for each
$i$, $d(i,i+1)=m-2$ and  for each $i$ there is a unique admissible
path with $m$ vertices between $i$ and $i+1$ .

 Now assume that $P: j_1,\ldots ,j_t$ ($t>m$) is an admissible path in
 $C_n$. So, $j_t> j_1+1$.

  If $n$ is odd, by the fact that
 $d(i,i+1)=m-2$ for each $i$, we conclude $j_1+1\in V(P)$ which is a
 contradiction.

 Assume that $n$ is even. If $j_1$ is odd, as above we conclude that $j_1+1\in
 V(P)$ which is the desired contradiction. If $j_1$ is even and $j_1+1\notin V(P)$, then $P'= j_1+1,
 j_1,\ldots ,j_{t-1}$ is a path with $t$ vertices. Since
 $d(j_1+1,j_1+2)=m-1$, $j_1+2\in V(P')$. So $j_t>j_1+2$ and
 $j_1+2\in V(P)$ again a contradiction.
\end{proof}

The next corollaries are the generalization of the fact that a
closed graph is chordal. These results are immediate consequences of
Proposition \ref{bridge} and Theorem \ref{cycle}.
\begin{Corollary}
If $G$ is an $m-$closed graph, then each cycle of $G$ with $2m-1$ or
more vertices has a chord.
\end{Corollary}

\begin{Corollary}
Let $G$ be an $m-$closed graph and $\ell=\max \{ t\ | \ \exists
{\text { an induced cycle with }}\ t\ {\text{vertices in }}\ G\}$.
If $\ell \geq 4$, then $m\geq\left\{
\begin{array}{ll}
    \frac{\ell}{2}+1 & \hbox{$\ell$ is even;} \\
    \frac{\ell+1}{2}+1 & \hbox{$\ell$ is odd.} \\
\end{array}%
\right.$
\end{Corollary}

A generalization of the notion of closed graph is {\it weakly closed
graph} which has been introduced in \cite{M1} Let G be a graph. $G$
is said to be weakly closed if there exists a labeling which
satisfies the following condition: for all $i, j$ such that $\{i,
j\} \in E(G)$, $i$ is adjacentable with $j$ (for the definition of
adjacentable see \cite[Definition 1.2]{M1}. The following theorem is
a characterization of weakly closed graphs.

\begin{Theorem}\label{weakly closed 1}\cite[Theorem 1.9]{M1} \\
Let $G$ be a graph. Then the following conditions are equivalent:
\begin{enumerate}
\item
 $G $ is weakly closed.
 \item There exists a labeling which satisfies
the following condition: for all$ i, j$ such that $\{i, j\} \in
E(G)$ and $j
> i + 1$, the following assertion holds: for all $i < k < j$, $\{i, k\}
\in E(G)$ or $\{k, j\}\in E(G)$.
\end{enumerate}
\end{Theorem}

In the following we relate the $m-$closed graphs to weakly closed
graphs.

\begin{Theorem}\label{weakly closed}
Let $G$ be a weakly closed graph. Then $G$ is $m-$closed for some
$m\leq 4$.
\end{Theorem}
\begin{proof}
Suppose that $G$ is a weakly closed graph on $[n]$. Then by Theorem
\ref{weakly closed 1}, for all $ i, j$ such that $\{i, j\} \in E(G)$
and $j
> i + 1$, the following assertion holds: for all $i < k < j$, $\{i, k\}
\in E(G)$ or $\{k, j\}\in E(G)$.

We prove that each admissible path of $G$ has at most $4$ vertices.
Assume to the contrary that there exists an admissible path $P:
i=i_1,i_2,\ldots, i_{m-1},i_m=j$ with $m\geq 5$ vertices.  Note that
$i<j$. If $i_2>j$, then $i<j<i_2$ and $\{i,i_2\}\in E(G)$. So
$\{i,j\}\in E(G)$ or $\{i_2,j\}\in E(G)$ which is a contradiction.
If $i_{m-1}<i$, then $i_{m-1}<i<j$ and $\{i_{m-1},j\}\in E(G)$. So
$\{i_{m-1},i\}\in E(G)$ or $\{i,j\}\in E(G)$. Again, it is a
contradiction. Therefore $i_2<j$ and $i_{m-1}>i$. Since $P$ is an
admissible path, we have $i_2<i$ and $i_{m-1}>j$.

Let $$t=\min\{r\ | 2<r\leq m-1, i_r>j\}.$$ So $i_{t-1}<i<j<i_t$ and
$\{i_{t-1},i_t\}\in E(G)$. If $t=3$, then $\{i_2,j\}\in E(G)$ or
$\{j,i_3\}\in E(G)$ which is impossible because $m\geq 5$ and $P$ is
an admissible path. If $t>3$, then $\{i_{t-1},i\}\in E(G)$ or
$\{i,i_t\}\in E(G)$. This case also is impossible since $P$ is an
admissible path.

So, in any case we get a contradiction. Thus $m\leq  4$ and the
result follows.
\end{proof}

Note that the converse of Theorem \ref{weakly closed} is not true
since $C_5$ is $4-$closed and not weakly closed.


\section{$3-$closed trees}

In the following we are going to characterize $3-$closed trees.

Let $G$ be a simple graph on the vertex set $[n]$ and $\mathcal{G}$
has no element of degree more than $3$, then $d(i,i+1)\leq 2$ for
each $i$. But the converse is not true in general. For example, let
$C$ be the cycle on the vertex set $[n]$ and with the edge set
$\{\{1,3\},\{3,4\},\{2,4\},\{2,5\},\{1,5\}\}$. Then for each $i$,
$d(i,i+1)\leq 2$ but $C$ is $4-$closed.

We recall that by \cite[Corollary 1.3]{HHKR}, a tree is a  closed
graph if and only if it is a path. Next result shows that
 a $3-$closed
labeling for a tree $T$ is a labeling in which
 $d(i,i+1)\leq 2$ for each $i$.

\begin{Theorem}\label{3-closed trees}
Let $T$ be a tree with $n$ vertices and assume that $T$ is not a
path. Then $T$ is $3-$closed if and only if there exists a labeling
for $V(T)$ such that $d(i,i+1)\leq 2$ for each $i$.
\end{Theorem}
\begin{proof}
Assume to the contrary that there exists a tree $T$ on the vertex
set $[n]$ such that $d(i,i+1)\leq 2$ for each $i$, and $T$ has an
admissible path of length at least $3$. Let
$$m-1=\max\{\ell(P)\ | \ P \ {\text {is an admissible path}}\}$$  and
$$i_1=\max\{ t\ | \ {\text {there exists an admissible path of
length}}\ m-1\ {\text{ starting from}} \ t\}.$$ Then $m>3$ and we
can consider an admissible path  like $P: i_1,i_2,\ldots , i_m$.
Since $T$ is a tree, $d(i_1,i_m)\geq 3$. So, $i_1+1\neq i_m$ which
shows that $i_1<i_1+1\leq i_m-1<i_m$. Therefore $i_1+1\notin
\{i_2,i_3,\ldots ,i_{m-1}\}$. Moreover, by $d(i_1,i_1+1)\leq 2$, one
of the following situations happens:

\vspace{2mm}

{\bf Case a:} $\{i_1,i_1+1\}\in E(T)$. In this case,
$i_1+1,i_1,\ldots , i_m$ is an admissible path of length $m$ which
is a contradiction by our choice of $m$.

\vspace{2mm}

{\bf Case b:} $\{i_1+1,i_2\}\in E(T)$. In this case,
$i_1+1,i_2,i_3,\ldots , i_m$ is an admissible path of length $m-1$
which is a contradiction by our choice  of $i_1$.

\vspace{2mm}

{\bf Case c:} There exists $j\in [n]\setminus \{i_2,\ldots ,i_m\}$
such that $i_1+1, j, i_1$ is a path. In this case, consider the path
$P': i_1+1, j , i_1, i_2, \ldots , i_m$. Since $\ell(P')=m+1$, by
our choice of $m$, $P'$ is not an admissible path. So,
$i_1<i_1+1<j<i_m$. It is easy to see that $P'': j, i_1, i_2,\ldots
,i_m$ is an admissible path of length $m$ which is again a
contradiction by our choice of $m$.
\end{proof}

\begin{Remark}
By Theorem \ref{3-closed trees}, a labeling of a tree $T$ is a
$3-$closed labeling if and only if $d(i,i+1)\leq 2$ for each $1\leq
i<n$. This is not true for an arbitrary $3-$closed graph. For
example, Let $G$ be a graph with $V(G)=\{v_1,\ldots, v_5\}$ and
$E(G)=\{\{v_1,v_2\},\{v_2,v_4\},\{v_1,v_3\},\{v_3,v_4\},\{v_2,v_5\}\}$.
 Then $G$ is a bipartite $3-$closed graph. If we assign $i$ to each
vertex  $v_i$, then $d(i,i+1)\leq 2$ for each $1\leq i<5$ but this
is not a $3-$closed labeling of $G$.
\end{Remark}
Next we give an example of a tree which is not $3-$closed.
\begin{Example}
Consider the following tree on $16$ vertices (Figure 1).
 $\ $
\begin{figure}[h!]
\centering \scalebox{0.6}{
\begin{tikzpicture}[vertex/.style={circle,draw=black,fill=black,thick,
inner sep=0pt,minimum size=2mm}] \node[label=60:$i_{16}$] (i16) at
(0,0) [vertex] {}; \node[label=60:$i_{15}$] (i15) at (144:2)
[vertex] {}; \node[label=60:$i_6$] (i6) at (0:2)  [vertex] {};
\node[label=0:$i_3$] (i3) at (72:2)  [vertex] {};
\node[label=60:$i_9$] (i9) at (-72:2) [vertex]  {};
\node[label=90:$i_{12}$] (i12) at (-144:2)  [vertex] {};
\node[label=60:$i_4$] (i4) at (12:4)  [vertex] {};
\node[label=60:$i_5$] (i5) at (-12:4)  [vertex] {};
\node[label=60:$i_2$] (i2) at (60:4)  [vertex] {};
\node[label=60:$i_1$] (i1) at (84:4)  [vertex] {};
\node[label=60:$i_{14}$] (i14) at (132:4)  [vertex] {};
\node[label=60:$i_{13}$] (i13) at (156:4) [vertex]  {};
\node[label=60:$i_{11}$] (i11) at (204:4) [vertex]  {};
\node[label=0:$i_{10}$] (i10) at (228:4) [vertex]  {};
\node[label=60:$i_8$] (i8) at (276:4)  [vertex] {};
\node[label=60:$i_7$] (i7) at (300:4)  [vertex] {}; \draw
(i16)--(i6)--(i4); \draw (i6)--(i5); \draw (i16)--(i12)--(i10);
\draw (i12)--(i11); \draw (i16)--(i9)--(i8); \draw (i9)--(i7); \draw
(i16)--(i15)--(i13); \draw (i15)--(i14); \draw (i16)--(i3)--(i1);
\draw (i3)--(i2);
\end{tikzpicture}}
\caption{}
\end{figure}

We prove that $T$ is not  $3-$closed. By contradiction assume that
there exists a labeling of $V(T)$ such that
\begin{equation}\label{distance}
 d(k,k+1)\leq 2 \ \ \ {\text {for all} }\  k\in \{1,\ldots ,15\}.
\end{equation}

Without loss of generality, we can assume that
$\{1,16\}\cap\{i_7,i_8,\ldots,i_{15}\}=\emptyset$. So,
\begin{equation}\label{t} \{i_j-1,i_j+1\}\subset \{1,2,\ldots
,16\} \ \ \ {\text {for all} } \ j\in \{7,8,\ldots , 15\}.
\end{equation}

If $i_7<i_8$  and they are not two successive integers, then by
(\ref{distance}) $\{i_7-1,i_7+1,i_8+1\}\subseteq \{i_9,i_{16}\}$
which is a contradiction. So, we can assume that $i_8=i_7+1$. By a
similar argument, we should also have, $i_{11}=i_{10}+1$ and
$i_{14}=i_{13}+1$.

Again, by  (\ref{distance}) and (\ref{t}) we can easily see that
$$i_{16}=i_7-1 \ {\text {or}} \ i_{16}=i_7+2,$$
and
$$i_{16}=i_{10}-1\  {\text {or}}\  i_{16}=i_{10}+2,$$
and
$$i_{16}=i_{13}-1\  {\text {or}} \ i_{16}=i_{13}+2.$$

So, $i_7=i_{10}$ or $i_7=i_{13}$ or $i_{10}=i_{13}$ which is a
contradiction.
\end{Example}

\begin{Definition}\label{bipartite}
Let $G$ be a graph on the vertex set $[n]$, we associate to $G$ a
bipartite graph $G^*$ where
$$V(G^*)=\{x_1,\ldots ,x_n\}\sqcup \{y_1,\ldots ,y_n\}, \ \
E(G^*)=\{x_iy_j\ | \ \{i,j\}\in E(G)\ {\text {and}}\ i<j\}.$$

Note that if $G$ is a closed graph, for a closed labeling of $G$,
$LT_{\prec}(J_G)=I(G^*)$ where $I(G^*)$ is the edge ideal of the
graph $G^*$.

Conversely, if $H$ is a bipartite graph on the vertex set
$\{x_1,\ldots ,x_n\}\sqcup \{y_1,\ldots ,y_n\}$ such that for each
$\{x_i,y_j\}\in E(H)$ we have $i<j$, then we can associate to $H$ a
simple graph $H_*$ on the vertex set $[n]$ in a natural way
($(H_*)^*=H$).
\end{Definition}

Note that if $T$ is a tree, then $T^*$ is also a tree. In the
following, we give a characterization of $3-$closed trees by means
of Definition \ref{bipartite}.

\begin{Theorem}\label{3-closed class}
Let $\mathcal{T}_n$ be the set of all bipartite graphs $H$ on the
vertex set $\{x_1,\ldots ,x_n\}\sqcup \{y_1,\ldots ,y_n\}$ with the
following properties:
\begin{enumerate}
\item $\{x_i,y_j\}\in E(H)\Longrightarrow i<j$.

\item for all $i\in \{1,\ldots , n-1\}$ one of the following
conditions holds:
\begin{itemize}
    \item $\{x_i, y_{i+1}\}\in E(H).$
    \item There exists $j>i+1, \ \{x_i,y_{j}\}, \{x_{i+1},y_j\}\in E(H).$
    \item There exists $j<i, \{x_j,y_i\}, \{x_j,y_{i+1}\}\in E(H).$
\end{itemize}
\item $|E(H)|=n-1$
\end{enumerate}
Then a tree with $n$ vertices is $3-$closed if and only if $T$ is
not a path and there exists $H\in \mathcal{T}_n$ such that $T\cong
H_*$.
\end{Theorem}
\begin{proof}
If $T$ is a $3-$closed graph on $[n]$, then, by Theorem
\ref{3-closed trees}, $d(i,i+1)\leq 2, \forall 1\leq i<n$. So $T^*$
satisfies condition 2. Since $|E(T)|=|E(T^*)|=n-1$, the conclusion
follows from the fact that $T=(T^*)_*$.

Conversely, if $H$ satisfies condition 1 then $H_*$ is defined and
is a graph on $[n]$.  By condition 2, in $H_*$, $d(i,i+1)\leq 2$ for
each $i$ and moreover $H_*$ is connected. Now since
$|E(H_*)|=n-1=|V(H_*)|-1$, $H_*$ is a tree. So, by Theorem
\ref{3-closed trees}, $H_*$ is a $3-$closed tree.
\end{proof}

In the next corollary, we find the number of elements of the reduced
 Gr\"{o}bner basis of a $3-$closed tree.

\begin{Corollary}\label{number of reduced}
Let $T$ be a tree on the vertex set $[n]$ and $d(i,i+1)\leq 2$ for
all $i\in\{1,\ldots ,n-1\}$. Then
$|\mathcal{G}|=n-1+\beta_{13}(I(T^*))$.
\end{Corollary}

\begin{proof}
Let $G$ be a simple graph on $[n]$ and $K_3(G)=$the number of
triangles of $G$. Then by \cite[Theorem 2.2]{SK},
$\beta_{13}(J_G)=2K_3(G)$. So, for an arbitrary tree $T$,
$\beta_{13}(J_T)=0$.

Now, if $d(i,i+1)\leq 2$, then by Theorem \ref{3-closed trees},
$LT_\prec(J_T)$ is generated in degrees $2$ and $3$. So,
$\beta_{23}(LT_\prec(J_T))=0$ and
$$\beta_{13}(LT_\prec(J_T))=\beta_{13}(\langle x_iy_j\ | \ i<j, \
\{i,j\}\in E(T)\rangle)=\beta_{13}(I(T^*)).$$

 Moreover, by \cite{Pe} the graded Betti numbers of
$J_T$ is obtained from the graded Betti numbers of $LT_\prec(J_T)$
by a sequence of consecutive cancelations. So
$$\beta_{03}(LT_\prec(J_T))=\beta_{13}(LT_\prec(J_T))=\beta_{13}(I(T^*))$$
and the conclusion follows.
\end{proof}

We remark that if $G$ is an arbitrary $3-$closed graph, for a
$3-$closed labeling,  the same argument as the proof of Corollary
\ref{number of reduced} shows that
$|\mathcal{G}|=|E(G)|+\beta_{13}(I(G^*))-2K_3(G).$


\section{BINOMIAL EDGE IDEALS OF CATERPILLAR TREES}
 In this section, we study the binomial edge ideals of caterpillar
 trees and some trees constructed from this kind of trees. First we
 recall its definition.

 \begin{Definition}
 A caterpillar tree is a tree $T$ with the property that it contains a path $P$  such that any
vertex of $T$ is either a vertex of $P$ or it is adjacent to a
vertex of $P$.
 \end{Definition}

\begin{figure}[h!]
\centering \scalebox{0.6}{
\begin{tikzpicture}[vertex/.style={circle,draw=black,fill=black,thick,
inner sep=0pt,minimum size=2mm}] \node[label=-90:$v_1$] (v1) at
(0,0) [vertex] {}; \node[label=-90:$v_2$] (v2) at (2,0) [vertex] {};
\node[label=-90:$v_3$] (v3) at (4,0) [vertex] {};
\node[label=-90:$v_4$] (v4) at (6,0) [vertex] {};
\node[label=-90:$v_5$] (v5) at (8,0) [vertex] {};
\node[label=-90:$v_6$] (v6) at (10,0) [vertex] {};
\node[label=-90:$v_7$] (v7) at (12,0) [vertex] {}; \node[label=-90:]
(b1) at (1.32,1.7) [vertex] {}; \node[label=-90:] (b2) at (2.68,1.7)
[vertex] {}; \node[label=-90:] (b3) at (6.68,1.7) [vertex] {};
\node[label=-90:] (b4) at (8.68,1.7) [vertex] {}; \node[label=-90:]
(b5) at (10.68,1.7) [vertex] {}; \node[label=-90:] (b6) at
(9.32,1.7) [vertex] {}; \draw (v2)--(b1); \draw (v2)--(b2); \draw
(v4)--(b3); \draw (v5)--(b4); \draw (v6)--(b5); \draw (v6)--(b6);
\draw (v1)--(v2)--(v3)--(v4)--(v5)--(v6)-- (v7);
\end{tikzpicture}}
\caption{}
\end{figure}

Note that  the path $P$ in the definition of a caterpillar tree is a
longest induced path of $T$ and we call it the central path of $T$.
Figure 2 is an example of a caterpillar tree with the central path
$P:v_1,v_2,\ldots ,v_7$.

 Caterpillar trees were first studied by Harary and
Schwenk \cite{HS}. These graphs have some applications in chemistry
and physics \cite{EI}.

Let $T$ be a caterpillar tree and $\ell$ be the length of its
longest induced path. By \cite[Theorem 1.1]{EHH}
$\depth(R/J_T)=|V(T)|+1$ and by \cite[Theorem 4.1]{ChDI}
$\reg(R/J_T)=\ell$. In the following we describe the minimal primary
decomposition of $J_T$. We recall that since $J_T$ is a radical
ideal, to know the minimal primary decomposition of $J_T$, it is
enough to characterize its minimal prime ideals.

\begin{Theorem}\label{primary of caterpillar}
Let $T$  be a caterpillar tree,  $ P: v_{1}, \ldots, v_{l} $ be the
central path of $T$ and $S\subset V(T)$. Then  $ P_{S}(T) $ is a
minimal prime ideal of $ J_{T} $ if and only if $ S=\emptyset $ or $
S=\lbrace v_{i_{1}}, \ldots, v_{i_{k}}\rbrace\subseteq \lbrace
v_{1}, \ldots, v_{l}\rbrace $ where $ 1<i_{1}<\cdots<i_{k}<l $
satisfy  the following conditions:
\begin{itemize}
\item
If $ \deg(v_{i_{j}})=2 $, then $ d(v_{i_{j}},v_{i_{j+1}})\geq 2 $
and $ d(v_{i_{j}},v_{i_{j-1}})\geq 2 $
\item
If $ \deg(v_{i_{j}})=3 $, then $ d(v_{i_{j}},v_{i_{j+1}})\geq 2 $ or
$ d(v_{i_{j}},v_{i_{j-1}})\geq 2 $.
\end{itemize}
\end{Theorem}

\begin{proof}
We prove that each prime ideal corresponding to a set $ S $, where $
S $ is satisfying in the mentioned conditions, is a minimal prime
ideal by
 induction on the number of vertices in the set $ S $.

For $ k=1 $ the statement is obvious. Now assume theorem is true for
each $S$ with $|S|=m$ and  $ S^{'}=\lbrace v_{i_{1}}, \ldots,
v_{i_{m+1}}\rbrace $ has the mentioned conditions. If
$S=\{v_{i_1},\ldots , v_{i_m}\}$, by induction hypothesis, $P_S(T)$
is a minimal prime ideal of $J_T$. Let $d=\deg(v_{i_{m+1}})$ and
$d'=\deg(v_{i_m})$.

 Depending on
$ d(v_{i_m},v_{i_{m+1}}) $, we distinguish the following cases:
\begin{itemize}
\item[\textbf{case 1:}]
$ d(v_{i_m},v_{i_{m+1}})\geq 2 $. In this case it is easy to see
that $ c(S^{'})=c(S)+d-1 $ and  for all $j\in\lbrace 1, \cdots,
m\rbrace,~~c(S^{'}\setminus\lbrace
v_{i_{j}}\rbrace)=c(S\setminus\lbrace v_{i_{j}}\rbrace)+d-1 $.
\item[\textbf{case 2:}]
$ d(v_{i_m},v_{i_{m+1}})=1 $.  In this case, $d\geq 3$ and $d'\geq
3$. A straightforward observation shows that $ c(S^{'})=c(S)+d-2 $
and for all $j\in\lbrace 1, \ldots, m-1\rbrace,~~ c(S^{'}\setminus
\lbrace v_{i_{j}}\rbrace)=c(S\setminus \lbrace v_{i_{j}}\rbrace)+d-2
$. Moreover for deleting the vertex $v_{i_m}$, one of the following
situations happens:
\begin{itemize}
\item[(a)]
$m=1$ or $ d(v_{i_{{m-1}}},v_{i_{m}})=2 $. One can see $
c(S^{'}\setminus \{v_{i_{m}}\})=c(S')-(d'-2)$.
\item[(b)]
$d'\geq 4$. In this case, $c(S')\geq c(S'\setminus
\{v_{i_m}\})+(d'-3)$.
\end{itemize}
\end{itemize}
It is obvious that in all of the above situations, $
c(S^{'}\setminus \lbrace v_{i_{j}}\rbrace)<c(S^{'})~~$ for all $j\in
\lbrace 1, \cdots, m+1\rbrace$.
 So, Lemma \ref{minimal prime} implies $ P_{S'}(T) $ is a minimal prime ideal of $ J_{T}
 $.\\

Now assume that  $S\subset V(T)$ is not as described in the theorem.
So, one of the following situation happens:
\begin{itemize}
\item[1)]
 There exists a vertex  $v$ of degree $1$ in $S$. In this case,  $c(S\setminus\lbrace v\rbrace)\geq
 c(S)$. So, by Lemma \ref{minimal prime}, $P_S(T)$ is not a minimal
 prime ideal of $J_T$.
\item[2)]
 For some $j$, $\deg(v_{i_j})=2$, and
 ($d(v_{i_j},v_{i_{j+1}})=1$ or  $d(v_{i_j},v_{i_{j-1}})=1$).
 Without loss of generality assume that $ d(v_{i_{j-1}},v_{i_{j}})=1 $. Since $ v_{i_{j-1}} $ and $ v_{i_{j}} $ are connected through just one edge,
  removing the vertex $ v_{i_{j}} $ doesn't change the number of connected components of $ T_{V(T)\setminus S} $, meaning that
   $ c(S\setminus\lbrace v_{i_{j}}\rbrace)=c(S) $. Again, by Lemma \ref{minimal prime}, $P_S(T)$ is not a minimal
 prime ideal of $J_T$.
\item[3)]
  For some $j$, $\deg(v_{i_j})=3$,
 $d(v_{i_j},v_{i_{j+1}})=1$ and $d(v_{i_j},v_{i_{j-1}})=1$. In this situation also
 straightforward observation shows that $ c(S\setminus\lbrace
 v_{i_{j}}\rbrace)=c(S)$. So, $P_S(T)$ is not a minimal
 prime ideal of $J_T$.
\end{itemize}
So the conclusion follows.
\end{proof}
For example, if $T$ is the caterpillar tree described in Figure 2,
then by Theorem \ref{primary of caterpillar}, it is easy to find all
minimal prime ideals of $J_T$ and see that $\dim(R/J_T)=19$.

%
%
%
%
%
%
Finally, we prove that  caterpillar trees and some trees constructed
by caterpillar trees are $3-$closed.

\begin{Theorem}\label{caterpilar trees}
 (a) Let $T$ be a caterpillar tree. Then $T$ is $3-$closed.

(b) Let $T=T_1\cup B\cup T_2$ where $T_1$ and $T_2$ are two
caterpillar trees and $B$ is a bridge between $T_1$ and $T_2$, and
the endpoints of $B$ are chosen from the  vertices of the central
paths of $T_1$ and $T_2$ respectively. Then $T$ is $3-$closed.

More generally,

(c)  Let  $T$ be a tree and $T=T_1\cup B\cup T_2$ where $T_1$, $T_2$
and $B$ are caterpillar trees, and the endpoints of the central path
of $B$ are chosen from the vertices of $T_1$ and $T_2$ respectively.
Then $T$ is $3-$closed.

\end{Theorem}
\begin{proof}
(a) Let $n=|T|$, it is enough to find a labeling of $V(T)$ such that
$d(i,i+1)\leq 2$ for each $1\leq i<n$.

Let $P:v_1,\ldots ,v_\ell$ be the central path of $T$, and for each
$1\leq j\leq \ell$, $ N'_{T}(v_{j})=N_{T}(v_j)\setminus V(P)$ is
determined the leaf neighbors of the vertex $ v_{j}. $

\vspace{2mm}

We do as follows:

\vspace{2mm}
\hspace{-0.4cm}label $ v_{1} $ as $ 1 $; $t=2$; $j=2$;\\
\textbf{While} $ j\leq \ell $ \textbf{do}\\
\hspace{0.7cm} label $v_{j}$ as $t$; $t=t+1$;\\
\hspace{0.7cm}$ S:=N'_{T}(v_{j}) $;\\
\hspace{0.7cm}\textbf{While} $ S\neq \emptyset $ \textbf{do};\\
\hspace{1.4cm}$ v $:= pick $ v\in S $ such that $ v $ is the rightmost leaf of $ v_{j} $;\\
\hspace{1.4cm}label $ v $ as $ t $;\\
\hspace{1.4cm}$ t=t+1 $; $ S=S\setminus \lbrace{ v\rbrace} $;\\
\hspace{0.7cm}\textbf{end};\\
\hspace{0.7cm}$ j=j+1 $;\\
\textbf{end}\\
\vspace{1mm}

It is easy to see that by this labeling of $V(T)$, $d(i,i+1)\leq 2$
for all $1\leq i< n$.

(b) By proposition \ref{bridge}, it is enough to show that for each
caterpillar tree $T$ and each vertex $v$ of its central path, there
exists a $3-$closed labeling in which $1$ is assigned  to $v$. This
fact follows from Algorithm 1.

(c) Without loss of generality, we can assume that the endpoints of
the central path of  $B$ are chosen from the vertices of the central
paths of $T_1$ and $T_2$ respectively.  Because if this is not the
case and for example $\{v\}=V(T_1)\cap V(B)$ where $v$ is not a
vertex of the central path of $T_1$, then there exists a vertex $w$
of the central path of $T_1$ such that $e=\{v,w\}\in E(T_1)$. So we
can replace $T_1$ with $T_1\setminus e$ and $B$ with $B\cup e$. We
can also assume that $E(T_1)$, $E(T_2)$ and $E(B)$ are pairwise
disjoint sets.

Let $v\in V(B)\cap V(T_1)$ and $w\in V(B)\cap V(T_2)$. By Algorithm
1, there exists a $3-$closed labeling of $V(T_1)$ that assigns
$n_1=|V(T_1)|$ to $v$. By part (a) of the proof there exists a
$3-$closed labeling of $V(B)$ with integers $n_1,\ldots ,
n_2=n_1+|V(B)|-1$  that assigns  $n_1$ to $v$ and $n_2$ to $w$.
Again by Algorithm 1 there exists a $3-$closed labeling of $V(T_2)$
with integers $n_2,\ldots , n_3=n_2+|V(T_2)|-1$ that associate $n_2$
to $w$. All together we get a $3-$closed labeling of $T$ and the
conclusion follows.
\end{proof}
 By \cite[Proposition 3.2]{M1}, a tree $T$ is weakly closed if and
 only if $T$ is a  caterpillar tree. So, by Theorem \ref{caterpilar
 trees}, If $T$ is a weakly closed graph, then
 $T $ is $3-$closed.

\section{Appendix}
In the following we introduce an algorithm to label the vertices of
a caterpillar tree $T$ with integers $1,\ldots ,n$ such that
$d(i,i+1)\leq 2$ for all $1\leq i<n$. Suppose that the central path
of $T$ is  $P:v_1,\ldots ,v_\ell$ and for each $1\leq j\leq \ell$, $
N'_{T}(v_{j})=N_{T}(v_j)\setminus V(P)$ is determined the leaf
neighbors of the vertex $ v_{j}. $

The algorithm works as follows. First a candidate for $ 1 $ is found
by choosing an arbitrary  vertex of the central path which is called
$ v_{i_{0}} $. We then go through the vertices in the central path.
If $ v_{i_{0}+1} $ has some leaf neighbors, we label them $ 2,
\ldots, t $ from right to left, and then we label $ v_{i_{0}+2} $ as
$ t+1 $. Otherwise we label $ v_{i_{0}+2} $ as 2. Then we set $
j=i_{0}+2 $ and this process is repeated for the next vertices of
the $ v_{j} $ until we reach the endpoint of $P$. In the return path
from $ v_{l} $ to $ v_{1} $ and then from $ v_{1} $ to $v_{i_0}$ the
similar process is repeated until every vertex is labeled.
$  $\\
\newpage
\begin{tabular}{l}
\hline \textbf{Algorithm 1}: labeling algorithm of caterpillars trees \\
\hline
\textbf{Input}: A caterpillar tree $T$ with the central path $P:v_1,\ldots ,v_\ell$.\\
\textbf{Output}: A $3-$closed labeling of $T$\\
$ v_{i_{0}} $:= one of the vertices on the central path;\\
$j:=i_{0}$;\ \
label $ v_{i_{0}} $ as $ 1 $;\ \  $t:=2$;\\
\textbf{While} $ j<\ell-1 $ \textbf{do}\\
\hspace{0.7cm}$ S:=N'_T(v_{j+1}) $;\\
\hspace{0.7cm}\textbf{While} $ S\neq \emptyset $ \textbf{do};\\
\hspace{1.4cm}$ v $:= pick $ v\in S $ such that $ v $ is the rightmost leaf of $ v_{j+1} $; label $ v $ as $ t $;\\
\hspace{1.4cm}$ t=t+1 $; $ S=S\setminus \lbrace{ v\rbrace} $;\\
\hspace{0.7cm}\textbf{end};\\
\hspace{0.7cm}label $ v_{j+2} $ as $ t $;\\
\hspace{0.7cm}$ j=j+2 $; $ t=t+1 $;\\
\textbf{end}\\
\textbf{If} $ j==\ell-1 $\\
\hspace{0.7cm}label $ v_{\ell} $ as $ t $;\\
\hspace{0.7cm}$ j=l $; $ t=t+1 $;\\
\textbf{Otherwise}\\
\hspace{0.7cm}label $ v_{\ell-1} $ as $ t $;\\
\hspace{0.7cm}$ j=\ell-1 $; $ t=t+1 $;\\
\textbf{end}\\
\textbf{While} $ j>2 $ \textbf{do}\\
\hspace{0.7cm}$ S:=N'_T(v_{j-1}) $;\\
\hspace{0.7cm}\textbf{While} $ S\neq \emptyset $ \textbf{do};\\
\hspace{1.4cm}$ v $:= pick $ v\in S $ such that $ v $ is the rightmost leaf of $ v_{j-1} $; label $ v $ as $ t $;\\
\hspace{1.4cm}$ t=t+1 $; $ S=S\setminus \lbrace{ v\rbrace} $;\\
\hspace{0.7cm}\textbf{end};\\
\hspace{0.7cm}label $ v_{j-2} $ as $ t $;\\
\hspace{0.7cm}$ j=j-2 $; $ t=t+1 $;\\
\textbf{end}\\
\textbf{If} $ j==2$ and $i_{0}>1 $\\
\hspace{0.7cm}label $ v_{1} $ as $ t $;\\
\hspace{0.7cm}$ j=1 $; $ t=t+1 $;\\
\textbf{Otherwise}\\
\hspace{0.7cm}\textbf{If} $i_0>2$\\
\hspace{1.4cm}label $ v_{2} $ as $ t $;\\
\hspace{1.4cm}$ j=2 $; $ t=t+1 $;\\
\hspace{0.7cm}\textbf{end};\\
\textbf{end };\\
\textbf{While} $ j<i_{0}-2 $ \textbf{do}\\
 \hspace{0.7cm}$ S:=N'_T(v_{j+1}) $;\\
\hspace{0.7cm}\textbf{While} $ S\neq \emptyset $ \textbf{do};\\
\hspace{1.4cm}$ v $:= pick $ v\in S $ such that $ v $ is the rightmost leaf of $ v_{j+1} $; label $ v $ as $ t $;\\
\hspace{1.4cm}$ t=t+1 $; $ S=S\setminus \lbrace{ v\rbrace} $;\\
\hspace{0.7cm}\textbf{end};\\
\hspace{0.7cm}label $ v_{j+2} $ as $ t $;\\
\hspace{0.7cm}$ j=j+2 $; $ t=t+1 $;\\
\textbf{end}\\
\textbf{If} $j==i_0-2$\\
\hspace{0.7cm}$S:=N'_T(v_{i_{0}-1})$;\\
\hspace{0.7cm}\textbf{While} $ S\neq \emptyset $ \textbf{do}\\
\hspace{1.4cm}$ v $:= pick $ v\in S $ such that $ v $ is the rightmost leaf of $ v_{i_0-1} $; label $ v $ as $ t $;\\
\hspace{1.4cm}$ t=t+1 $; $ S=S\setminus \lbrace{ v\rbrace} $;\\
\hspace{0.7cm}\textbf{end};\\
\textbf{end};\\
\hline
\end{tabular}

\begin{Remark}\label{algorithm remark}
If one wants to give a $3-$closed labeling to a caterpillar tree $T$
in such  a way that $1$ is assigned to $v\in N'_T(v_{i_0})$ for some
$1<i_0<\ell$, it is enough to  label $v$ as $1$, $v_{i_0}$ as $2$,
set $N'_T(v_{i_0})=N'_T(v_{i_0})\setminus \{v\}$ and start with
$t:=3$ instead of $t:=2$.

Moreover, if one wants to give a $3-$closed labeling to a
caterpillar tree $T$ in such a way that $n=|V(T)|$ is assigned  to
an arbitrary vertex $v$, it is enough to apply Algorithm 1, by
labeling $v$ as 1 and at the end changing the label $i$ of each
vertex to $n-i+1$.
\end{Remark}

\begin{Example}
Here, we give an example of a labeled caterpillar tree using
Algorithm 1. Note that $12$ is the label of $v_1$, $11$ is the label
of $v_2$, $1$ is the label of $v_3$ and so on .
 $\ $
\begin{figure}[h!]
\centering \scalebox{0.6}{
\begin{tikzpicture}[vertex/.style={circle,draw=black,fill=black,thick,
inner sep=0pt,minimum size=2mm}] \node[label=-90:$12$] (v1) at (0,0)
[vertex] {}; \node[label=-90:$11$] (v2) at (2,0) [vertex] {};
\node[label=-90:$1$] (v3) at (4,0) [vertex] {}; \node[label=-90:$7$]
(v4) at (6,0) [vertex] {}; \node[label=-90:$3$] (v5) at (8,0)
[vertex] {}; \node[label=-90:$6$] (v6) at (10,0) [vertex] {};
\node[label=-90:$5$] (v7) at (12,0) [vertex] {};
\node[label=-90:$10$] (v8) at (2.5,2) [vertex] {};
\node[label=-100:$9$] (v9) at (4,2) [vertex] {};
\node[label=-90:$8$] (v10) at (5.5,2) [vertex] {};
\node[label=-90:$2$] (v11) at (7.5,2) [vertex] {};
\node[label=-90:$4$] (v12) at (11.5,2) [vertex] {};
\draw(v1)--(v2)--(v3)--(v8)--(v3)--(v4)--(v5)--(v6)--(v7); \draw
(v3)--(v9); \draw (v3)--(v10); \draw (v4)--(v11); \draw (v6)--(v12);
\end{tikzpicture}}
\caption{}
\end{figure}

Finally, we give an example of a $3-$closed tree described in
Theorem \ref{caterpilar trees}(part b). Note that the labeling is
given by Algorithm 1, and Proposition \ref{bridge}.

\begin{figure}[h!]
\centering \scalebox{0.6}{
\begin{tikzpicture}
[vertex/.style={circle,draw=black,fill=black,thick, inner
sep=0pt,minimum size=2mm}] \node[label=90:$1$] (v1) at (0,0)
[vertex] {}; \node[label=90:$2$] (v2) at (2,0) [vertex] {};
\node[label=100:$12$] (v3) at (4,0) [vertex] {}; \node[label=90:$6$]
(v4) at (6,0) [vertex] {}; \node[label=90:$10$] (v5) at (8,0)
[vertex] {}; \node[label=90:$7$] (v6) at (10,0) [vertex] {};
\node[label=90:$8$] (v7) at (12,0) [vertex] {}; \node[label=90:$3$]
(v8) at (2.5,-2) [vertex] {}; \node[label=100:$4$] (v9) at (4,-2)
[vertex] {}; \node[label=90:$5$] (v10) at (5.5,-2) [vertex] {};
\node[label=90:$11$] (v11) at (7.5,-2) [vertex] {};
\node[label=90:$9$] (v12) at (11.5,-2) [vertex] {};

\node[label=-90:$21$] (v13) at (0,2.5) [vertex] {};
\node[label=-90:$20$] (v14) at (2,2.5) [vertex] {};
\node[label=-100:$13$] (v15) at (4,2.5) [vertex] {};
\node[label=-90:$17$] (v16) at (6,2.5) [vertex] {};
\node[label=-90:$15$] (v17) at (8,2.5) [vertex] {};
\node[label=-90:$16$] (v18) at (10,2.5) [vertex] {};
\node[label=90:$23$] (v19) at (1.5,4.5) [vertex] {};
\node[label=90:$22$] (v20) at (2.5,4.5) [vertex] {};
\node[label=90:$19$] (v21) at (3.5,4.5) [vertex] {};
\node[label=90:$18$] (v22) at (4.5,4.5) [vertex] {};
\node[label=90:$14$] (v23) at (7.5,4.5) [vertex] {}; \draw
(v1)--(v2)--(v3)--(v8)--(v3)--(v4)--(v5)--(v6)--(v7); \draw
(v3)--(v9); \draw (v3)--(v10); \draw (v4)--(v11); \draw (v6)--(v12);
\draw (v13)--(v14)--(v15)--(v16)--(v17)--(v18); \draw (v14)--(v19);
\draw (v14)--(v20); \draw (v15)--(v21); \draw (v15)--(v22); \draw
(v16)--(v23); \draw (v3)--(v15);
\end{tikzpicture}}
\caption{}
\end{figure}
\end{Example}
\subsection*{Acknowledgements}
 We are very grateful to the anonymous
referee for suggesting us Theorem \ref{weakly closed}.
\bibliographystyle{amsalpha}

 \end{document}